\documentclass[11pt]{amsart}

\usepackage{amsmath,amssymb,amsthm,mathtools}
\usepackage[margin=0.88in]{geometry}
\usepackage{tikz}
\usetikzlibrary{calc,fit,positioning,backgrounds}
\usepackage{microtype}
\usepackage{hyperref}

\title{New Perspectives on the Unimodality of Domination Polynomials}
\author{Mohamed Omar}
\address{Department of Mathematics \& Statistics. York University. 4700 Keele St. Toronto, Canada. M3J 1P3}
\email{omarmo@yorku.ca}

\date{\today}


\newtheorem{theorem}{Theorem}
\newtheorem{conjecture}[theorem]{Conjecture}
\newtheorem{lemma}[theorem]{Lemma}
\newtheorem{proposition}[theorem]{Proposition}
\newtheorem{corollary}[theorem]{Corollary}
\newtheorem{definition}[theorem]{Definition}

\newtheorem{example}[theorem]{Example}

\numberwithin{equation}{section}

\DeclareMathOperator{\dist}{dist}
\DeclareMathOperator{\argu}{arg}

\newcommand{\VV}{V(G)}
\newcommand{\EE}{E(G)}
\newcommand{\HH}{\mathcal{H}}
\newcommand{\bbC}{\mathbb{C}}
\newcommand{\bbR}{\mathbb{R}}

\begin{document}

\begin{abstract}
The domination polynomial of a graph $G$ is given by $D(G,x)=\sum_{k=0}^{n} d_k(G)x^k$ where $d_k(G)$ records the number of $k$-element dominating sets in $G$. A conjecture of Alikhani and Peng asserts that these polynomials have unimodal coefficient sequences. We develop three complementary perspectives that strengthen existing tools for resolving the conjecture. First, we view dominating sets as transversals of the closed neighborhood hypergraph. Motivated by the relationship between the unimodality of a polynomial and its roots, we use this perspective to expand on known root phenomena for domination polynomials. In particular, we obtain a bound on the modulus of domination roots that is linear in the maximum degree of a graph, improving related exponential bounds of Bencs, Csikv\'{a}ri and Regts. The hypergraph viewpoint also yields explicit combinatorial formulas for top coefficients of $D(G,x)$, extending formulas in the literature and offering fruitful ground for combinatorial approaches to the unimodality conjecture. Second, we strengthen the coefficient-ratio method of Beaton and Brown. This includes tightening their inequalities, and combining a union bound for non-dominating $k$-element sets with an overlap correction based on spanning trees. This produces a new parameter $\tau_k(G)$ measuring maximal pairwise neighborhood overlap and yields an overlap-corrected sufficient criterion for unimodality. Third, we prove that the domination polynomial of threshold graphs are log-concave, and hence unimodal, by a planar network argument from total positivity. This offers a new tactic for resolving the unimodality of hereditary graph classes.
\end{abstract}

\maketitle

\section{Introduction}

Let $G$ be a finite simple undirected graph with vertex set $\VV$, and $n=|\VV|$. A set $S\subseteq \VV$ is \emph{dominating} if every vertex is either in $S$ or adjacent to a vertex of $S$. Let $d_k(G)$ denote the number of dominating sets of size $k$. Alikhani and Peng \cite{AlikhaniPeng2014} introduced the  \emph{domination polynomial} given by
\begin{equation}\label{eq:defDomPoly}
D(G,x) = \sum_{k=0}^{n} d_k(G) x^k,
\end{equation}
which records dominating sets of varying sizes in its coefficients. Domination polynomials have since been studied from several points of view, including recurrence relations and splitting formulas \cite{KotekTittmannTrinks2012}, domination equivalence and reconstruction questions \cite{BeatonPaths2019}, and the distribution of roots \cite{BeatonBrown2022,BeatonBrownRealRoots2021,BencsCsikvariRegts2021}. 

The domination polynomial packages dominating set information in a way that parallels several other graph polynomials. The independence polynomial counts independent sets by size, while the domination polynomial counts hitting sets for closed neighborhoods. Matching polynomials and independence polynomials often exhibit strong analytic phenomena (real-rootedness, log-concavity, sharp zero-free regions) that can be leveraged for coefficient questions. Domination polynomials are less rigid, but their coefficient sequences still appear to have robust shape. One reason is that dominating sets form a large and highly correlated family: in most graphs, once a set is moderately large, it is likely to dominate many vertices simultaneously. The normalized coefficients $r_k(G)$ from \eqref{eq:rkDef} make this probabilistic intuition precise.

Another source of structure comes from graph operations. For disjoint unions one has $D(G_1\cup G_2,x)=D(G_1,x)\,D(G_2,x)$, and for joins there are similar explicit formulas, see \cite{KotekTittmannTrinks2012}.
Such identities allow one to build large families of graphs with controllable domination polynomials from smaller pieces. However, unimodality is not automatically preserved under all natural operations, and controlling the interaction between pieces remains delicate. This is one reason it is helpful to have several complementary approaches to the unimodality conjecture of Alikhani and Peng \cite{AlikhaniPeng2014}, which is one of the long standing conjectures on these polynomials.

\begin{conjecture}[\cite{AlikhaniPeng2014}]\label{conj:unimodal}
For every finite graph $G$, the domination sequence $(d_k(G))_{k=0}^{n}$ is unimodal.
\end{conjecture}

In many settings unimodality is a shadow of a stronger property such as log-concavity or real-rootedness, but domination polynomials do not generally enjoy either: even when $G$ is a tree, domination polynomials can have nonreal roots and need not be log-concave \cite{BeatonBrown2022}. Existing proofs of unimodality therefore tend to rely on more delicate combinatorial or probabilistic structure, and even basic families such as trees remain challenging.

The current evidence for Conjecture~\ref{conj:unimodal} comes from several directions. First, the conjecture holds for a number of dense families where domination is ``easy'' in the sense that a random set of size about $n/2$ almost surely dominates. This includes graphs of sufficiently large minimum degree \cite{BeatonBrown2022}. Second, the conjecture has been verified for certain highly structured graph classes, including complete multipartite graphs \cite{BeatonBrown2022} and other families with explicit domination recurrences. It has also been verified for many special families of trees, graphs with maximum degree at most $400$, and persists on other highly specialized families (see \cite{Burcroff2023}). On the other hand, sparse graphs are genuinely subtle. For example, for trees the domination number can be linear in $n$, and the family of dominating sets reflects fine local structure (leaves, supports, and branching). From this point of view, Conjecture~\ref{conj:unimodal} for trees is reminiscent of the long-standing conjecture that the independence polynomial of a tree is unimodal. In both cases, the central coefficients are governed by many competing local configurations.

It is therefore natural to look for methods that are sensitive to local neighborhood structure but still yield global control. The three perspectives developed here aim to do exactly that: Perspective~1 translates domination into a hypergraph independence problem and imports global analytic information. Perspective~2 measures the failure of domination by local avoidance events but corrects for their overlaps. Perspective~3 imports theory from total positivity to establish log-concavity and hence unimodality. Before detailing the results from each perspective, we detail pertinent relevant results in the literature.

One of the most general results to date is due to Beaton and Brown \cite{BeatonBrown2022}. They introduced the normalized coefficients
\begin{equation}\label{eq:rkDef}
r_k(G):=\frac{d_k(G)}{\binom{n}{k}},
\end{equation}
interpretable as the probability that a uniformly random $k$-element subset of $\VV$ is dominating. They proved that $r_k(G)$ is nondecreasing in $k$. Since $\binom{n}{k}$ is increasing for $k\le \lfloor n/2\rfloor$, this implies that $d_k(G)$ is increasing for $k\le \lfloor n/2\rfloor$.

The same paper provides a criterion that forces the tail of the domination sequence to be nonincreasing. In one convenient form, if $k\ge n/2$ and $r_k(G)\ \ge \frac{n-k}{k+1}$, then $d_{k}(G)\ge d_{k+1}(G)\ge \cdots \ge d_n(G)$. Together with the monotonicity of $r_k(G)$, it follows that verifying this inequality at $k=\lceil n/2\rceil$ implies unimodality with mode $\lceil n/2\rceil$. 

A different but closely related line of work concerns the \emph{roots} of domination polynomials. Since $D(G,x)$ has nonnegative coefficients it has no positive real roots, but it can have complex roots with positive real part and with arguments close to $0$. On the real axis, the picture is now essentially complete: Beaton and Brown proved that the closure of the set of real domination roots is $(-\infty,0]$ \cite{BeatonBrownRealRoots2021}. For complex roots, Bencs, Csikv\'ari, and Regts used Wagner's weighted subgraph counting polynomial to obtain a uniform disk containing all roots when the maximum degree is bounded, and they related domination roots to roots of edge cover polynomials \cite{BencsCsikvariRegts2021}. These questions are important due to the relationship between log-concavity, unimodality, and roots of a polynomial.

The purpose of this paper is to develop three complementary perspectives that interact naturally with these themes. The first perspective is that dominating sets can be interpreted in terms of independent sets in an associated hypergraph. For a graph $G$, consider the hypergraph $\HH_G$ whose hyperedges are the closed neighborhoods $N[v]$ for $v \in \VV$. Dominating sets then are precisely complements of hypergraph independent sets in $\HH_G$. This gives the identity $D(G,x)=x^n I(\HH_G,1/x),$ where $I(\HH_G,x)$ is the hypergraph independence polynomial. While the identity is elementary, it is powerful because it allows us to transport recent sharp zero-free results for bounded degree hypergraph independence polynomials to domination polynomials. We use this to obtain a bound on domination roots that is linear in the maximum degree of $G$, and to prove the existence of a uniform root-free neighborhood of a positive ray for all graphs with bounded maximum degree. On the coefficient side, the identity shows that $d_{n-k}(G)$ counts $k$-element independent sets in $\HH_G$, and for small $k$ this yields explicit formulas in terms of local low degree configurations. As a representative result we give a full expression for $d_{n-3}(G)$ valid for all graphs, keeping track of isolated vertices and small components.

The second perspective is an extension of the Beaton-Brown ratio method. Since $r_k(G)$ is the probability that a random $k$-element set is dominating, one wants to bound the probability of failure. A basic union bound produces a quantity $\sigma_k(G)$ such that $r_k(G)\ge 1-\sigma_k(G)$. We use this to strengthen and generalize the degree lower bound condition of Beaton and Brown that guarantees unimodality. We then observe that this bound ignores overlaps among the events that avoid specific neighborhoods; in graphs with repeated neighborhoods those overlaps are substantial. To mitigate for this, we introduce an overlap correction based on spanning trees, encoded in a new term $\tau_k(G)$, yielding $r_k(G)\ge 1-\sigma_k(G)+\tau_k(G)$. This leads to a new sufficient criterion for unimodality at $k=\lceil n/2\rceil$. 

The third perspective shows how unimodality can be inherited from existing theory in enumerative combinatorics. We show that polynomials that are sums of shifted binomials whose exponents follow a prescribed ``staircase" structure are log-concave and hence unimodal. This is established by using a planar network argument from total positivity, exploiting classic Lindstr\"{o}m-Gessel-Viennot theory. We use this in turn to prove that a ubiquitous and well-studied class of graphs, threshold graphs, have unimodal domination polynomials.

The paper is organized into three main sections corresponding to these perspectives, preceded by a brief preliminaries section and followed by a conclusion emphasizing open directions for extending these methods.

\section{Preliminaries}

We introduce preliminaries that will be used throughout. For $v\in\VV$, denote by $N[v]$ the closed neighborhood of $v$, which consists of $v$ and all vertices adjacent to $v$. The minimum and maximum degrees of a graph $G$ are $\delta(G)$ and $\Delta(G)$, respectively. We write $i(G)$ for the number of isolated vertices and $\ell(G)$ for the number of leaves. 

Throughout our work, we use two results from \cite{BeatonBrown2022} as black boxes. These statements will be used repeatedly; we refer to (ii) as the Beaton-Brown tail criterion.
\begin{proposition}\cite{BeatonBrown2022}\label{rem:BB}
Let $G$ be a graph on $n$ vertices and define $r_k(G)$ as in \eqref{eq:rkDef}.

\smallskip\noindent
(i) The sequence $r_0(G),r_1(G),\dots,r_n(G)$ is nondecreasing.

\smallskip\noindent
(ii) If $k\ge n/2$ and $r_k(G)\ge (n-k)/(k+1)$ then $d_k(G)\ge d_{k+1}(G)\ge\cdots\ge d_n(G)$. In particular, if the inequality holds at $k=\lceil n/2\rceil$, then $D(G,x)$ is unimodal with mode $\lceil n/2\rceil$.

\smallskip
\end{proposition}

We also recall basic hypergraph terminology. A \emph{hypergraph} $\HH$ consists of a finite vertex set $V(\HH)$ and a collection of hyperedges $E(\HH)\subseteq 2^{V(\HH)}$. An \emph{independent set} in $\HH$ is a set $U\subseteq V(\HH)$ that contains no hyperedge as a subset. The \emph{independence polynomial} is
\[
I(\HH,x)=\sum_{\substack{U\subseteq V(\HH)\\ U\ \text{independent}}} x^{|U|}.
\]
The \emph{maximum degree} $\Delta(\HH)$ is the maximum number of hyperedges containing a single vertex. For a graph $G$, a corresponding hypergraph of relevance in this article is the \emph{closed neighborhood hypergraph} $\HH_G$. This hypergraph has vertex set and edge set given by
\[
V(\HH_G)=\VV,\qquad E(\HH_G)=\{N[v]:v\in\VV\}.
\]
Dominating sets in $G$ are exactly the transversals of $\HH_G$, and this reduction is standard in domination-set enumeration \cite{Kante2014}. More explicitly we see that taking complements, $S\subseteq\VV$ is dominating if and only if $U=\VV\setminus S$ is independent in $\HH_G$ because a dominating set $S$ must intersect $N[v]$ nontrivially for every $v \in \VV$. Summing over all such $U$ gives
\begin{equation}\label{eq:reciprocity}
D(G,x)=\sum_{S \text{ dominating in } G} x^{|S|} = \sum_{U \text{ independent in } \HH_G} x^{n-|U|} = x^n I(\HH_G,1/x).
\end{equation}

\begin{example}
Let $G$ be the path graph on vertex set $\VV=\{1,2,3,4\}$ with edge set $\EE=\{12,23,34\}$. The closed neighborhoods of $G$ give rise to the closed neighborhood hypergraph $\HH_G$ with vertex set $V(\HH_G)=\VV$ and hyperedge set $E(\HH_G)=\{12,34,123,234\}$. The dominating sets in $G$ are 
\[
\{1,3\},\{1,4\},\{2,3\},\{2,4\},\{1,2,3\},\{1,2,4\},\{1,3,4\},\{2,3,4\},\{1,2,3,4\}
\]
so $D(G,x)=x^4+4x^3+4x^2$. The independent sets in $\HH_G$ are
\[
\emptyset, \{1\}, \{2\}, \{3\}, \{4\}, \{1,3\}, \{1,4\},\{2,3\},\{2,4\}
\]
so $I(\HH_G)=4x^2+4x+1$. We see that $x^4I(\HH_G,1/x)=x^4 \cdot \left(\frac{4}{x^2}+\frac{4}{x}+1\right)=D(G,x)$.
\end{example}

The reciprocity identity \eqref{eq:reciprocity} governs our findings in Perspective 1. We make an important note that will be useful throughout: a vertex $u\in\VV=V(\HH_G)$ lies in a hyperedge $N[v]$ if and only if $v\in N[u]$, so $u$ belongs to exactly $\deg_G(u)+1$ hyperedges. Consequently
\begin{equation}\label{eq:maxDegHG}
\Delta(\HH_G)=\Delta(G)+1.
\end{equation}

\section{Perspective 1: the closed neighborhood hypergraph}

\subsection{Top coefficients from low degree structure.}

The identity \eqref{eq:reciprocity} implies that $d_{n-k}(G)$ counts the $k$-element independent sets of $\HH_G$, equivalently the complements of $(n-k)$-element dominating sets. For small $k$, independence in $\HH_G$ is governed by small hyperedges, hence by low degree vertices in $G$. It is this that allows us to access top coefficients of $D(G,x)$ readily.

Two immediate consequences are worth recording. If $0\le k\le \delta(G)$ then every $k$-element subset $U\subseteq\VV$ is independent in $\HH_G$, because every hyperedge has size at least $\delta(G)+1$. Therefore
\begin{equation}\label{eq:topkBinomial}
d_{n-k}(G)=\binom{n}{k}\qquad (0\le k\le \delta(G)).
\end{equation}
Similarly, for $k=\delta(G)+1$ a $(\delta(G)+1)$-element subset fails to be independent if and only if it is exactly a hyperedge $N[v]$ with $\deg(v)=\delta(G)$. Thus
\begin{equation}\label{eq:topDeltaPlusOne}
d_{n-\delta(G)-1}(G)=\binom{n}{\delta(G)+1}-\bigl|\{N[v]:\deg(v)=\delta(G)\}\bigr|.
\end{equation}
These identities are straightforward consequences of \eqref{eq:reciprocity}.

More generally, for a fixed integer $k$, the coefficient $d_{n-k}(G)$ can be computed by inclusion-exclusion over those vertices $v$ with $|N[v]|\le k$, since only those hyperedges can be contained in a $k$-element subset of $\VV$. In graph language this means that $d_{n-k}(G)$ is determined entirely by the local structure around vertices of degree at most $k-1$, together with the way their closed neighborhoods overlap. Theorem~\ref{thm:dn-3} is the first nontrivial instance of this principle. One conceptual advantage of phrasing the computation in $\HH_G$ is that the relevant configurations are naturally organized by the sizes of the hyperedges rather than by ad hoc graph cases.

For $k=1$ and $k=2$ one can easily recover the familiar formulas
\begin{equation}\label{eq:dn-1}
d_{n-1}(G)=n-i(G),
\end{equation}
and
\begin{equation}\label{eq:dn-2}
d_{n-2}(G)=\binom{n}{2}-\Bigl(i(G)(n-1)-\binom{i(G)}{2}-e_1(G)\Bigr),
\end{equation}
where $e_1(G)$ denotes the number of edges incident to a leaf. These appear, for example, in \cite{BeatonPaths2019} (see also the survey discussion there for additional sources).

We now give an explicit expression for $d_{n-3}(G)$ in terms of configurations of degree at most $2$, valid for all graphs. Earlier formulas typically assumed the graph has no isolated vertices and no $K_2$ components; see \cite{BeatonPaths2019}. 

\begin{theorem}\label{thm:dn-3}
Let $G$ be a graph on $n$ vertices. Then
\begin{align}
d_{n-3}(G) &= \binom{n}{3}
-\Bigl(i(G)\binom{n-1}{2}+\ell(G)(n-2)+n_2(G)\Bigr)\notag\\
&\quad+\Bigl(\binom{i(G)}{2}(n-2)+i(G)\ell(G)+e_2(G)(n-2)+s(G)+m(G)+t_{22}(G)\Bigr)\notag\\
&\quad-\Bigl(\binom{i(G)}{3}+i(G)e_2(G)+p_3(G)+c_3(G)\Bigr),\label{eq:dn-3-formula}
\end{align}
where:
$n_2(G)$ is the number of vertices of degree $2$,
$e_2(G)$ is the number of connected components isomorphic to $K_2$,
$s(G)$ is the number of unordered pairs of leaves with a common neighbor,
$m(G)$ is the number of edges joining a leaf to a degree $2$ vertex,
$t_{22}(G)$ is the number of unordered pairs of degree $2$ vertices with equal closed neighborhood,
$p_3(G)$ is the number of connected components isomorphic to the path $P_3$,
and $c_3(G)$ is the number of connected components isomorphic to $K_3$.
\end{theorem}

\begin{proof}
By \eqref{eq:reciprocity}, $d_{n-3}(G)$ is the number of $3$-element independent subsets in the hypergraph $\HH_G$. Equivalently, if $U\subseteq\VV$ with $|U|=3$ and $U$ is called \emph{bad} when it contains a hyperedge $N[v]$ as a subset, then
\[
d_{n-3}(G)=\binom{n}{3}-|\{U\subseteq\VV: U \text{ is bad }\}|.
\]
A hyperedge $N[v]$ has size $\deg(v)+1$, so $N[v]\subseteq U$ with $|U|=3$ forces $\deg(v)\le 2$. For each vertex $v$ with $\deg(v)\le 2$, let $A_v$ denote the event that $N[v]\subseteq U$. Then $U$ is bad if and only if at least one event $A_v$ occurs. We compute the number of bad $U$ by inclusion-exclusion over the events $A_v$.

If $v$ is isolated, then $N[v]=\{v\}$ and $A_v$ occurs precisely when $v\in U$. For each isolated $v$ there are $\binom{n-1}{2}$ choices for $U$, hence isolated vertices contribute $i(G)\binom{n-1}{2}$ to the count of bad sets. If $v$ is a leaf, with unique neighbor $w$, then $N[v]=\{v,w\}$ and $A_v$ occurs precisely when $\{v,w\}\subseteq U$, leaving $n-2$ choices for the third element. Summing over all leaves contributes $\ell(G)(n-2)$. If $v$ has degree $2$ with neighbors $u$ and $w$, then $N[v]=\{u,v,w\}$ and there is exactly one $3$-element set containing it, so each vertex of degree $2$ contributes $1$, giving $n_2(G)$. The total single event contribution is therefore
\[
i(G)\binom{n-1}{2}+\ell(G)(n-2)+n_2(G).
\]

Next we adjust this count by considering $A_u\cap A_v$ for distinct vertices $u$ and $v$ of degree at most $2$. Because $|U|=3$, this intersection can occur only when $|N[u]\cup N[v]|\le 3$, and in each admissible case the resulting $U$ is forced or nearly forced. If $u$ and $v$ are distinct isolated vertices, then $N[u]\cup N[v]=\{u,v\}$ and $A_u\cap A_v$ occurs precisely when $\{u,v\}\subseteq U$. There are $n-2$ choices for the third vertex, contributing $\binom{i(G)}{2}(n-2)$ over all such pairs. If $u$ is isolated and $v$ is a leaf, then $A_u\cap A_v$ forces $U$ to contain $u$ and the edge $N[v]$, hence $U$ is uniquely determined. The number of choices is $i(G)\ell(G)$. If $u$ and $v$ are leaves, then $|N[u]\cup N[v]|\le 3$ forces either that they are adjacent (forming a $K_2$ component) or that they share a common neighbor. If they are adjacent then $N[u]\cup N[v]=\{u,v\}$ and $A_u\cap A_v$ occurs when $U$ contains $u$ and $v$; the third element is arbitrary among the remaining $n-2$ vertices. Summing over $K_2$ components contributes $e_2(G)(n-2)$. If they are not adjacent but share a common neighbor $w$, then $N[u]\cup N[v]=\{u,v,w\}$ and $A_u\cap A_v$ forces $U=\{u,v,w\}$ uniquely. The number of such unordered pairs of leaves is $s(G)$. If one vertex is a leaf and the other has degree $2$, then for $|N[u]\cup N[v]|\le 3$ to hold the leaf must be adjacent to the vertex of degree $2$. In this case the union of neighborhoods equals the closed neighborhood of the vertex of degree $2$ and again forces $U$ uniquely. Each edge joining a leaf to a vertex of degree $2$ yields one such intersection, contributing $m(G)$. Finally, if $u$ and $v$ are both vertices with degree $2$, then $A_u\cap A_v$ can occur only if $N[u]=N[v]$, in which case $U$ is forced to equal this common neighborhood. The number of unordered pairs of degree $2$ vertices with identical closed neighborhood is $t_{22}(G)$. No other pair of vertices can yield $A_u\cap A_v$, because any other combination forces $|N[u]\cup N[v]|\ge 4$. Summing the admissible cases, the total contribution of all pairwise intersections is
\[
\binom{i(G)}{2}(n-2)+i(G)\ell(G)+e_2(G)(n-2)+s(G)+m(G)+t_{22}(G).
\]

Finally we adjust for triple intersections $A_u\cap A_v\cap A_w$. This can occur only when $U=N[u]\cup N[v]\cup N[w]$ has size at most $3$, and therefore when the three vertices form a union of connected components in $G$ whose closed neighborhoods stay within the triple. There are exactly four possibilities. If $u,v,w$ are all isolated, then the intersection holds for the unique set $U=\{u,v,w\}$, contributing $\binom{i(G)}{3}$ to the count. If two of the vertices form a $K_2$ component and the third is isolated, then the intersection corresponds to choosing a $K_2$ component and an isolated vertex, contributing $i(G)e_2(G)$. If $\{u,v,w\}$ forms a $P_3$ component, then all three closed neighborhoods are contained in $\{u,v,w\}$ and the intersection contributes $p_3(G)$. If $\{u,v,w\}$ forms a $K_3$ component, the same holds and contributes $c_3(G)$. No other triple intersections occur, because any other configuration forces at least one closed neighborhood to contain a vertex outside $\{u,v,w\}$.

Applying inclusion--exclusion yields \eqref{eq:dn-3-formula}.
\end{proof}

The same method extends in principle to $d_{n-4}(G)$ and beyond, but the number of admissible local configurations grows quickly. For $d_{n-4}$ one already encounters interactions among degree $3$ neighborhoods, and a tractable closed formula typically requires structural restrictions (such as excluding isolated vertices and $K_2$ components); see \cite{BeatonPaths2019} for a detailed discussion of such formulas and their use in domination equivalence.

\subsection{Roots via zero-free regions for $I(\HH,x)$.}

We now turn to roots. The identity \eqref{eq:reciprocity} reduces domination roots to zeros of a hypergraph independence polynomial. Recent work gives sharp, uniform zero-free regions for $I(\HH,x)$ in terms of the hypergraph maximum degree \cite{BencsBuys2025}. Using \eqref{eq:maxDegHG}, these results transfer to bounds on roots of domination polynomials in terms of $\Delta(G)$.

Our first result gives a modulus bound that is linear in $\Delta(G)$ for domination roots. For comparison, Bencs, Csikv\'ari, and Regts proved that if $G$ has no isolated vertices and $\Delta(G)\le \Delta$, then all roots of $D(G,x)$ lie in the disk $|z|\le 2^{\Delta+1}$ with radius bounded by an exponential in $\Delta$ \cite{BencsCsikvariRegts2021}. 

\begin{theorem}\label{thm:linearRootBound}
Let $G$ be a graph with maximum degree $\Delta=\Delta(G)\ge 1$. If $z$ is a complex root of $D(G,x)$, then $|z|<e(\Delta+1).$
\end{theorem}

\begin{proof}
Let $z\neq 0$ satisfy $D(G,z)=0$. Bencs and Buys proved that if $\HH$ is a hypergraph with maximum degree at most $M$, then $I(\HH,\lambda)\neq 0$ whenever
\begin{equation}\label{eq:shearerDisk}
|\lambda|<\lambda_s(M):=\frac{(M-1)^{M-1}}{M^{M}},
\end{equation}
and that this radius is optimal \cite{BencsBuys2025}. By \eqref{eq:maxDegHG}, the hypergraph $\HH_G$ has maximum degree $\Delta(\HH_G)=\Delta+1$, so applying the Bencs and Buys result with $M=\Delta+1$ gives that $I(\HH_G,\lambda)\neq 0$ for all $|\lambda|<\lambda_s(\Delta+1)=\Delta^{\Delta}/(\Delta+1)^{\Delta+1}$. By \eqref{eq:reciprocity}, since $z$ is a root of $D(G,z)=0$, $1/z$ is a root of $I(\HH_G,x)=0$, so \eqref{eq:shearerDisk} implies $\left|\frac{1}{z}\right| \geq \frac{\Delta^{\Delta}}{(\Delta+1)^{\Delta+1}},$ and therefore $|z| \leq \left(1+\frac{1}{\Delta}\right)^\Delta \cdot (1+\Delta)$. It is classical that $\left(1+\frac{1}{\Delta}\right)^\Delta<e$ for $\Delta\ge 1$, and the result follows.
\end{proof}

We next prove that for bounded degree graphs, domination roots avoid a uniform neighborhood of a positive ray beyond an explicit threshold. This is a complex analytic strengthening of the trivial observation that there are no positive real roots: it rules out roots approaching the positive real axis with modulus bounded away from zero, uniformly over all graphs of bounded maximum degree.

\begin{theorem}\label{thm:rayZeroFree}
Fix an integer $\Delta\ge 2$ and define $T_{\Delta}=\frac{(\Delta-1)^{\Delta+1}}{\Delta^{\Delta}}.$ There exists an open set $\Omega_{\Delta}\subseteq \bbC$ containing the ray $(T_{\Delta},\infty)$ such that for every graph $G$ with $\Delta(G)\le \Delta$ one has $D(G,z)\neq 0$ for all $z\in \Omega_{\Delta}$.
\end{theorem}

\begin{proof}
Let $M=\Delta+1$. Bencs and Buys proved that for every integer $M\ge 3$ there exists an open neighborhood $U_M$ of the interval $(0,\lambda_c(M))$, where $\lambda_c(M)=\frac{(M-1)^{M-1}}{(M-2)^{M}},$ such that $I(\HH,\lambda)\neq 0$ for all $\lambda\in U_M$ and all hypergraphs $\HH$ with $\Delta(\HH)\le M$ \cite{BencsBuys2025}. Since $\Delta\ge 2$ implies $M=\Delta+1\ge 3$, we may apply this with $M=\Delta+1$. Writing $L_{\Delta}:=\lambda_c(\Delta+1)=\Delta^{\Delta}/(\Delta-1)^{\Delta+1}$, we have $L_{\Delta}=1/T_{\Delta}$. Define
\[
\Omega_{\Delta}:=\{z\in\bbC\setminus\{0\}:1/z\in U_M\}.
\]
Since inversion is continuous on $\bbC\setminus\{0\}$, the set $\Omega_{\Delta}$ is open. Because $U_M$ contains $(0,L_{\Delta})$, the set $\Omega_{\Delta}$ contains $(T_{\Delta},\infty)$. Now let $G$ satisfy $\Delta(G)\le \Delta$. Then $\Delta(\HH_G)\le \Delta+1$ by \eqref{eq:maxDegHG}. If $z\in\Omega_{\Delta}$ then $1/z\in U_M$, hence $I(\HH_G,1/z)\neq 0$. By \eqref{eq:reciprocity}, $D(G,z)=z^{n}I(\HH_G,1/z)\neq 0$.
\end{proof}

The next theorem extracts a uniform angular separation around the positive real axis beyond any fixed radius $r>T_{\Delta}$. It is a compactness argument made quantitative via the stability of inversion near the positive real axis.

\begin{theorem}\label{thm:angularSeparation}
Fix $\Delta\ge 2$ and $r>T_{\Delta}$. There exists $\theta=\theta(\Delta,r)>0$ such that for every graph $G$ with $\Delta(G)\le \Delta$ the domination polynomial $D(G,x)$ has no root $z$ satisfying $|z|\ge r$ and $|\argu(z)|<\theta$.
\end{theorem}

\begin{proof}
Let $M=\Delta+1$ and $U_M$ be the open set from the proof of Theorem~\ref{thm:rayZeroFree}. Since $r>T_{\Delta}$, we have $1/r<L_{\Delta}$, where $L_{\Delta}=1/T_{\Delta}$. Consider the compact interval
\[
K:=\left[0,\frac{1}{r}\right]\subseteq \bbR.
\]
Because $U_M$ is open and contains $K$, there exists $\varepsilon>0$ such that the open $\varepsilon$-neighborhood of $K$ is contained in $U_M$. In other words, if $\lambda\in\bbC$ satisfies $\dist(\lambda,K)<\varepsilon$, then $\lambda\in U_M$. Now choose $\theta>0$ such that
\begin{equation}\label{eq:thetaChoice}
\frac{2}{r}\sin(\theta/2)<\varepsilon.
\end{equation}
We show that for any $z\in\bbC$ satisfying $|z|\ge r$ and $|\argu(z)|<\theta$, $D(G,z) \neq 0$. Write $z=|z|e^{i\phi}$ with $|\phi|<\theta$. Then $\frac{1}{z}=\frac{1}{|z|}e^{-i\phi}.$ Let $t:=1/|z|$, so $0<t\le 1/r$ and hence $t\in K$. We estimate
\[
\left|\frac{1}{z}-t\right|
=t\left|e^{-i\phi}-1\right|
=t\cdot 2\sin(|\phi|/2)
\le \frac{2}{|z|}\sin(\theta/2)
\le \frac{2}{r}\sin(\theta/2)
<\varepsilon
\]
by \eqref{eq:thetaChoice}. Thus $\dist(1/z,K)<\varepsilon$, so $1/z\in U_M$. Equivalently $z\in\Omega_{\Delta}$, where $\Omega_{\Delta}$ is as in Theorem~\ref{thm:rayZeroFree}.

Now let $G$ satisfy $\Delta(G)\le \Delta$. By Theorem~\ref{thm:rayZeroFree}, $D(G,z)\neq 0$ for all $z\in\Omega_{\Delta}$. In particular, $D(G,z)\neq 0$ for all $z$ with $|z|\ge r$ and $|\argu(z)|<\theta$.
\end{proof}

Theorems~\ref{thm:linearRootBound} through \ref{thm:angularSeparation} highlight a general theme: domination polynomials inherit zero-free regions from hypergraph independence polynomials through the transformation \eqref{eq:reciprocity}. The results of \cite{BencsBuys2025} are optimal in several senses, so the explicit thresholds in Theorems~\ref{thm:linearRootBound} and \ref{thm:rayZeroFree} are unlikely to be significantly improved without incorporating additional structure of the specific closed neighborhood hypergraphs that arise from graphs. Identifying graph classes whose closed neighborhood hypergraphs have stronger analytic properties is an attractive direction.

\section{Perspective 2: overlap-corrected inequalities for unimodality}

This section develops lower bounds on $r_k(G)$ by bounding the number of \emph{non-dominating} $k$-element subsets of $\VV$. The novelty is an overlap correction that strengthens the naive union bound in graphs with repeated or highly overlapping neighborhoods. One of the two key statistics controlling this behavior is the following:
\begin{definition}\label{def:sigma_k}
For $0 \leq k \leq n$, define
\begin{equation}\label{eq:sigmaDef} 
\sigma_k(G):=\sum_{v\in\VV}\frac{\binom{n-|N[v]|}{k}}{\binom{n}{k}},
\end{equation}
\end{definition}

Now fix a graph $G$ on $n$ vertices and an integer $0\le k\le n$. For each $v\in\VV$ define
\[
A_v:=\{S\subseteq\VV:|S|=k,\ S\cap N[v]=\varnothing\}.
\]
If $S\in A_v$, then $S$ is not a dominating set because it does not contain $v$ or any of its neighbors. Consequently, the family of non-dominating $k$-element sets is contained in $\bigcup_{v\in\VV} A_v$. A direct union bound on $\bigcup_{v\in\VV} A_v$ yields the following inequality.

\begin{proposition}\label{prop:unionBound}
For every graph $G$ on $n$ vertices and every $k$, 
\[
\displaystyle d_k(G)\ge\binom{n}{k}-\sum_{v\in\VV} \binom{n-|N[v]|}{k}.
\] 
Equivalently, $r_k(G)\ge 1-\sigma_k(G)$.
\end{proposition}

\begin{proof}
Every non-dominating $k$-element set lies in $\bigcup_{v\in\VV} A_v$, hence $d_k(G) \ge \binom{n}{k}-\left|\bigcup_{v\in\VV}A_v\right|.$ For each $v$, the sets in $A_v$ are precisely the $k$-element subsets of $\VV\setminus N[v]$, so $|A_v|=\binom{n-|N[v]|}{k}$. Substituting together with $|\bigcup_{v \in \VV} A_v| \leq \sum_{v \in \VV} |A_v|$ completes the proof.
\end{proof}

The quantity $\sigma_k(G)$ depends on the entire multiset of neighborhood sizes. It can be bounded in terms of a simpler degree sequence proxy
\begin{equation}\label{eq:sigmaSmallDef}
\sigma(G):=\sum_{v\in\VV}2^{-\deg(v)-1}
\end{equation}
as established in the following proposition.

\begin{proposition}\label{prop:sigmaCompare}
Let $G$ be a graph on $n$ vertices and set $k=\lfloor n/2\rfloor$. Then $\sigma_k(G)\le 2\sigma(G)$.
\end{proposition}

\begin{proof}
Fix $m$ with $1\le m\le n$. We claim
\begin{equation}\label{eq:binomRatioBound}
\frac{\binom{n-m}{k}}{\binom{n}{k}}\le 2^{1-m}.
\end{equation}
Assume $\binom{n-m}{k}\neq 0$ (otherwise the desired inequality is immediate). Using factorials and canceling yields
\[
\frac{\binom{n-m}{k}}{\binom{n}{k}}
=\frac{(n-m)!(n-k)!}{n!(n-m-k)!}
=\prod_{i=0}^{m-1}\frac{n-k-i}{\,n-i\,}.
\]
If $n$ is even, then $n=2t$ and $k=t$, so each factor equals $(t-i)/(2t-i)\le \tfrac12$, hence the product is at most $2^{-m}\le 2^{1-m}$. If $n$ is odd, then $n=2t+1$ and $k=t$, so the factors are $(t+1-i)/(2t+1-i)$; for $i\ge 1$ we have $(t+1-i)/(2t+1-i)\le \tfrac12$, while the remaining factor $(t+1)/(2t+1)\le 1$, so for $m\ge 2$ the product is at most $1\cdot 2^{-(m-1)}=2^{1-m}$, and for $m=1$ we have $(t+1)/(2t+1)\le 1=2^{0}$. In all cases, $\binom{n-m}{k}/\binom{n}{k}\le 2^{1-m}$ as claimed. Now applying \eqref{eq:binomRatioBound} with $m=|N[v]|=\deg(v)+1$ gives $\binom{n-|N[v]|}{k}/\binom{n}{k} \le 2^{-\deg(v)}$. Summing over $v$ yields $\sigma_k(G)\le \sum_{v \in \VV} 2^{-\deg(v)}=2\sum_{v \in \VV} 2^{-\deg(v)-1}=2\sigma(G)$.
\end{proof}

Combining Proposition~\ref{prop:unionBound} with the Beaton-Brown criterion yields unimodality whenever $\sigma_k(G)$ is sufficiently small at $k\approx n/2$. The next proposition is a convenient degree sequence sufficient condition.

\begin{proposition}\label{prop:sigmaCriterion}
Let $G$ be a graph on $n\ge 6$ vertices and set $k=\lceil n/2\rceil$. If $\sigma(G)\le 1/(n+2)$, then $D(G,x)$ is unimodal with mode $k$.
\end{proposition}

\begin{proof}
Let $k_0=\lfloor n/2\rfloor$. By Proposition~\ref{prop:unionBound} and Proposition~\ref{prop:sigmaCompare},
\[
1-r_{k_0}(G)\le \sigma_{k_0}(G)\le 2\sigma(G)\le \frac{2}{n+2}.
\]
By monotonicity of $r_k(G)$ (Proposition~\ref{rem:BB}(i)), $r_k(G)\ge r_{k_0}(G)\ge 1-\frac{2}{n+2}=\frac{n}{n+2}$. A direct check shows that for $n\ge 6$ and $k=\lceil n/2\rceil$ one has $\frac{n}{n+2}\ge \frac{n-k}{k+1}$. Therefore the Beaton-Brown tail criterion applies and yields unimodality with mode $k$.
\end{proof}

One advantage of Proposition~\ref{prop:sigmaCriterion} is that it is sensitive to the entire degree sequence and is stable under adding universal vertices (that is, vertices $v$ for which $N[v]=\VV$).

\begin{corollary}\label{cor:universalVertices}
Let $G$ be a graph on $n\ge 6$ vertices with $m$ universal vertices, and suppose every other vertex has degree at least $\delta$. If
\[
(n-m)\cdot 2^{-\delta-1}\le \frac{1}{n+2},
\]
then $D(G,x)$ is unimodal with mode $\lceil n/2\rceil$.
\end{corollary}

\begin{proof}
Each universal vertex contributes $2^{-(n-1)-1}=2^{-n}$ to $\sigma(G)$, while each non-universal vertex contributes at most $2^{-\delta-1}$. Thus
\[
\sigma(G)\le m\cdot 2^{-n}+(n-m)\cdot 2^{-\delta-1}\le (n-m)\cdot 2^{-\delta-1}+\frac{m}{2^n}.
\]
The hypothesis implies $\sigma(G)\le 1/(n+2)$, so Proposition~\ref{prop:sigmaCriterion} yields the claim.
\end{proof}

As a baseline comparison, we refine the result of Beaton and Brown proving unimodality when $\delta(G)\ge 2\log_2 n$ \cite{BeatonBrown2022}.

\begin{proposition}\label{prop:mindegLog}
Let $G$ be a graph on $n\ge 4$ vertices. If
\[
\delta(G)\ \ge\ \log_2(n(n+2)) - 2,
\]
then $D(G,x)$ is unimodal with mode $\lceil n/2\rceil$.
\end{proposition}

\begin{proof}
Let $k=\lceil n/2\rceil$. By Proposition~\ref{prop:unionBound}, $1-r_k(G)\le \sigma_k(G)\le n\cdot \frac{\binom{n-\delta(G)-1}{k}}{\binom{n}{k}}.$ Let $m=\delta(G)+1$. Using $\binom{n-m}{k}/\binom{n}{k}=\binom{n-k}{m}/\binom{n}{m}$ and bounding as in the proof of Proposition~\ref{prop:sigmaCompare} yields $\frac{\binom{n-m}{k}}{\binom{n}{k}}\le 2^{-m}.$ Hence $1-r_k(G)\le n\cdot 2^{-\delta(G)-1}$. Under the hypothesis $\delta(G)\ge \log_2(n(n+2))-2$, this gives $1-r_k(G)\le \frac{2}{n+2}$ and so $r_k(G)\ge \frac{n}{n+2}\ge \frac{n-k}{k+1}$. The Beaton-Brown tail criterion yields unimodality with mode $k$.
\end{proof}
The union bound in Proposition~\ref{prop:unionBound} ignores overlaps among the events $A_v$. We now introduce an overlap correction based on spanning trees, which is particularly effective when many vertices have similar closed neighborhoods. For $0\le k\le n$, define
\begin{equation}\label{eq:tauDef}
\tau_k(G):=\max_{T}\ \sum_{uv\in E(T)}\frac{\binom{n-|N[u]\cup N[v]|}{k}}{\binom{n}{k}},
\end{equation}
where the maximum ranges over all spanning trees $T$ on the vertex set $\VV$. The edges $uv$ of $T$ are not required to be edges of $G$; the parameter is measuring overlap among the families $A_v$, and for this purpose only the vertex set matters. It is helpful to view $\tau_k(G)$ as a combinatorial optimization problem.
Define weights on unordered pairs $\{u,v\}$ by $w_k(u,v):=\frac{\binom{n-|N[u]\cup N[v]|}{k}}{\binom{n}{k}}=\frac{|A_u\cap A_v|}{\binom{n}{k}}.$ Then $\tau_k(G)$ is the maximum total weight of a spanning tree in the complete graph on $\VV$ with these weights.
In particular, $\tau_k(G)$ can be computed in polynomial time by a maximum spanning tree algorithm.

The extreme case that motivates the definition is when two vertices have identical closed neighborhoods. If $N[u]=N[v]$ then $A_u=A_v$ and therefore $|A_u\cap A_v|=|A_u|$. If a graph contains a large class of vertices with identical closed neighborhood, the union bound counts the same event many times, while the spanning tree subtraction in Proposition~\ref{prop:overlapBound} can cancel all but one copy, making the bound essentially sharp on that twin class.

\begin{proposition}\label{prop:overlapBound}
For every graph $G$ on $n$ vertices and every $k$, $r_k(G)\ge1-\sigma_k(G)+\tau_k(G)$.
\end{proposition}

\begin{proof}
We use the notation $A_v$ from Proposition~\ref{prop:unionBound}. Fix any spanning tree $T$ on $\VV$ and choose a leaf $\ell$ of $T$ with unique neighbor $w$ in $T$. Let $U=\bigcup_{v\in \VV\setminus\{\ell\}} A_v$. Then $\left|\bigcup_{v\in\VV}A_v\right|=|U|+|A_\ell|-|U\cap A_\ell|.$ Since $A_w\subseteq U$, we have $U\cap A_\ell\supseteq A_w\cap A_\ell$, and therefore $\left|\bigcup_{v\in\VV}A_v\right|\le |U|+|A_\ell|-|A_w\cap A_\ell|.$ Applying the same argument to $U$ along the spanning tree obtained from $T$ by deleting $\ell$, and iterating until all vertices are deleted, yields $\left|\bigcup_{v\in\VV}A_v\right| \le \sum_{v\in\VV}|A_v|-\sum_{uv\in E(T)}|A_u\cap A_v|.$ This inequality holds for every spanning tree $T$, hence
\begin{equation}\label{eq:treeinequality}
\left|\bigcup_{v\in\VV}A_v\right|
\le \sum_{v\in\VV}|A_v|-\max_T\sum_{uv\in E(T)}|A_u\cap A_v|.
\end{equation}
Now, as in Proposition~\ref{prop:unionBound}, $d_k(G)\ge \binom{n}{k}-\left|\bigcup_{v\in\VV}A_v\right|.$ For each $v$ we have $|A_v|=\binom{n-|N[v]|}{k}$, and for each pair $u,v$ we have $|A_u\cap A_v|=\binom{n-|N[u]\cup N[v]|}{k}$. Substituting into the inequality in \eqref{eq:treeinequality} and dividing by $\binom{n}{k}$ gives $r_k(G)\ge 1-\sigma_k(G)+\tau_k(G)$.
\end{proof}

Combining Proposition~\ref{prop:overlapBound} with the Beaton-Brown tail criterion gives an overlap-corrected sufficient condition for unimodality at $k=\lceil n/2\rceil$.

\begin{theorem}\label{thm:sigmaTauUnimodality}
Let $G$ be a graph on $n$ vertices and let $k=\lceil n/2\rceil$. If
\begin{equation}\label{eq:sigmaTauCriterion}
\sigma_k(G)-\tau_k(G)\ \le\ \frac{2k+1-n}{k+1},
\end{equation}
then $D(G,x)$ is unimodal with mode $k$.
\end{theorem}

\begin{proof}
By Proposition~\ref{prop:overlapBound}, $r_k(G)\ge 1-\sigma_k(G)+\tau_k(G)$. Rearranging \eqref{eq:sigmaTauCriterion} gives $1-\sigma_k(G)+\tau_k(G)\ \ge\ \frac{n-k}{k+1},$ so $r_k(G)\ge (n-k)/(k+1)$. The Beaton-Brown tail criterion implies $d_k(G)\ge d_{k+1}(G)\ge\cdots\ge d_n(G)$, and together with the monotonicity of $r_k(G)$ this yields unimodality with mode $k$.
\end{proof}
We now give a concrete family of graphs were our phenomena apply. The key property is that any graph in the family contains a large twin class inside a clique, so the spanning tree subtraction encoded by $\tau_k(G)$ cancels essentially the entire union bound contribution of that class (as discussed following \eqref{eq:tauDef}), leaving a single binomial ratio that can be controlled using \eqref{eq:binomRatioBound}.

In that light, fix integers $s\ge 2$ and $N\ge 1$. Let $C$ and $Q$ be disjoint cliques on $N$ and $s$ vertices, respectively. Choose distinguished vertices $c_0\in C$ and $q_0\in Q$, and form a graph $G=G_{N,s}$ by adding the edges $c_0q$ for all $q\in Q\setminus\{q_0\}$ and write $n:=|V(G_{N,s})|=N+s$ (see Figure~\ref{fig:GNs-blob} for an illustration).


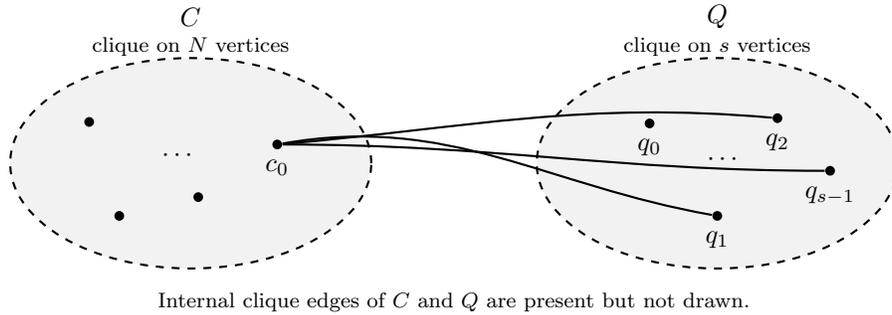
\begin{figure}[ht]
\centering
\begin{tikzpicture}[
  every node/.style={font=\small},
  vertex/.style={circle,fill=black,inner sep=1.3pt},
  special/.style={circle,draw,very thick,fill=white,minimum size=6mm,inner sep=0pt},
  edge/.style={thick}
]

\draw[dashed,thick,fill=gray!10] (0,0) ellipse (2.4 and 1.4);
\node at (0,1.95) {$C$};
\node[font=\scriptsize] at (0,1.55) {clique on $N$ vertices};

\draw[dashed,thick,fill=gray!10] (7,0) ellipse (2.4 and 1.4);
\node at (7,1.95) {$Q$};
\node[font=\scriptsize] at (7,1.55) {clique on $s$ vertices};

\node[vertex, label=below:$c_0$] (c0) at (1.15,0.25) {};
\node[vertex] (c1) at (-1.35,0.55) {};
\node[vertex] (c2) at (-0.95,-0.70) {};
\node[vertex] (c3) at (0.10,-0.45) {};
\node at (-0.15,0.10) {$\cdots$};

\node[vertex, label=below:$q_0$] (q0) at (6.10,0.53) {};
\node[vertex, label=below:$q_1$] (q1) at (7.00,-0.70) {};
\node[vertex, label=below:$q_2$] (q2) at (7.80,0.60) {};
\node[vertex, label=below:$q_{s-1}$] (qs) at (8.50,-0.10) {};
\node at (7.10,0.05) {$\cdots$};

\draw[edge] (c0) to[out=10,in=170] (q1);
\draw[edge] (c0) to[out=5,in=175]  (q2);
\draw[edge] (c0) to[out=0,in=180]  (qs);

\node[font=\scriptsize] at (3.5,-1.85)
{Internal clique edges of $C$ and $Q$ are present but not drawn.};

\end{tikzpicture}
\caption{$G_{N,s}$: two cliques $C$ and $Q$ with distinguished vertices $c_0\in C$ and $q_0\in Q$,
and edges $c_0q$ for every $q\in Q\setminus\{q_0\}$.}
\label{fig:GNs-blob}
\end{figure}


Begin by evaluating $\sigma_k(G)$ from definition \eqref{eq:sigmaDef}. For any $c\in C\setminus\{c_0\}$ we have $N[c]=C$, hence $n-|N[c]|=s<k$ and so $\binom{n-|N[c]|}{k}=0$ by our binomial convention. The vertex $c_0$ satisfies $N[c_0]=\VV\setminus\{q_0\}$, giving $n-|N[c_0]|=1<k$ and therefore contributing $0$ as well. On the $Q$ side, the vertex $q_0$ has $N[q_0]=Q$, so $n-|N[q_0]|=N$, while every $q\in Q\setminus\{q_0\}$ has $N[q]=Q\cup\{c_0\}$, so $n-|N[q]|=N-1$. Summing the resulting terms in \eqref{eq:sigmaDef} yields $\sigma_k(G)=\left(\binom{N}{k}+(s-1)\binom{N-1}{k}\right)/\binom{n}{k}$.

Now for $\tau_k(G)$, use the weight interpretation following \eqref{eq:tauDef}. If $u\in C\setminus\{c_0\}$ then $N[u]=C$, so $N[u]\cup N[v]$ contains $C$ for every $v\in\VV$, and therefore the complement $\VV\setminus(N[u]\cup N[v])$ is contained in $Q$ and has size at most $s<k$, forcing $w_k(u,v)=\binom{n-|N[u]\cup N[v]|}{k}/\binom{n}{k}=0$ for every $v$. Likewise if $u=c_0$ then $N[c_0]=\VV\setminus\{q_0\}$ implies $n-|N[c_0]\cup N[v]|\le 1<k$ and again $w_k(c_0,v)=0$ for all $v$. Hence the only potentially nonzero weights occur between pairs of vertices in $Q$. For any distinct $q,q'\in Q\setminus\{q_0\}$ we have $N[q]=N[q']=Q\cup\{c_0\}$, so $n-|N[q]\cup N[q']|=N-1$; also $N[q_0]\subseteq N[q]$ for $q\in Q\setminus\{q_0\}$, so $n-|N[q_0]\cup N[q]|=N-1$ as well. Thus every edge in the complete graph on vertex set $Q$ has weight $\binom{N-1}{k}/\binom{n}{k}$, while every edge incident to a vertex in $C$ in the complete graph on $\VV$ has weight $0$. A spanning tree on $\VV$ can contain at most $s-1$ edges with both ends in $Q$, and taking any spanning tree on $Q$ together with $N$ additional edges connecting the vertices of $C$ to $Q$ achieves $s-1$ such edges, so the maximum total weight is $(s-1)\binom{N-1}{k}/\binom{n}{k}$, and hence this is $\tau_k(G)$. It follows that $\sigma_k(G)-\tau_k(G)=\frac{\binom{N}{k}}{\binom{n}{k}}.$

Now for each integer $\ell\ge 4$, we claim the graph $G_{N,\ell}$ with $N=2^\ell-\ell-2$ has unimodal domination polynomial with mode $k=n/2$ where $n=2^{\ell}-2$. Since $\ell<k$ for $\ell\ge 4$, we have $\sigma_k(G_{N,\ell})-\tau_k(G_{N,\ell})=\binom{n-\ell}{k}/\binom{n}{k}$. Because $n$ is even, we have $k=\lfloor n/2\rfloor$, so \eqref{eq:binomRatioBound} applied with $m=\ell$ yields $\sigma_k(G_{N,\ell})-\tau_k(G_{N,\ell})\le 2^{1-\ell}$. On the other hand, since $n=2k$ the right-hand side of \eqref{eq:sigmaTauCriterion} is $(2k+1-n)/(k+1)=1/(k+1)$, and with $n=2^\ell-2$ we have $k+1=2^{\ell-1}$, hence $1/(k+1)=2^{1-\ell}$. Therefore \eqref{eq:sigmaTauCriterion} holds, and Theorem~\ref{thm:sigmaTauUnimodality} then implies $D(G_{N,\ell},x)$ is unimodal with mode $k$.

\section{Perspective 3: Total Positivity}
Our last perspective shows that inheriting theory from enumerative combinatorics is a fruitful endeavor for uncovering the unimodality of domination polynomials. We start by showing a broad class of polynomials have log-concave, and hence unimodal, domination sequences. This is done by appealing to the theory of total positivity, particularly appealing to a classical result of Lindstr\"{o}m-Gessel-Viennot on positivity and lattice paths in the plane. We then see how a ubiquitous class of graphs, threshold graphs, have domination polynomials in this broad class of polynomials, and therefore have unimodal domination polynomials.

\begin{lemma}\label{lem:unimodal_staircase_binomial_sum}
Let $\alpha_1 \ge \alpha_2 \ge \cdots \ge \alpha_m \ge 1$ and $0 \le \beta_1 < \beta_2 < \cdots < \beta_m$ be integers such that
\[
\alpha_{r+1}+\beta_{r+1}=\alpha_r+\beta_r+1\qquad\text{for all }r=1,\dots,m-1.
\]
Define $P(x) := \sum_{r=1}^m x^{\alpha_r}(1+x)^{\beta_r}.$ Then the coefficient sequence of $P(x)$ is unimodal.
\end{lemma}

\begin{proof}
Set $n_r := \alpha_r+\beta_r$ for $1 \leq r \leq m$. The hypothesis implies $n_{r+1}=n_r+1$, and hence $n_r=n_1+(r-1)$, for all $r$. In particular $N:=n_m=n_1+(m-1).$ Since $\deg\bigl(x^{\alpha_r}(1+x)^{\beta_r}\bigr)=n_r$, it follows that $\deg P = N$ because the top degree term $x^{\alpha_m+\beta_m}$ from $x^{\alpha_m}(1+x)^{\beta_m}$ has no cancellation. Writing $P(x)=\sum_{k=0}^{N} c_k x^k,$ we expand each summand:
\[
x^{\alpha_r}(1+x)^{\beta_r}
= x^{\alpha_r}\sum_{j=0}^{\beta_r}\binom{\beta_r}{j}x^j
= \sum_{j=0}^{\beta_r}\binom{\beta_r}{j}x^{\alpha_r+j}.
\]
With the convention $\binom{n}{t}=0$ when $t\notin\{0,1,\dots,n\}$, we have for every $0\le k\le N$,
\begin{equation}\label{eq:ck-sum}
c_k := [x^k]P(x)= \sum_{r=1}^m \binom{\beta_r}{\,k-\alpha_r\,}.
\end{equation}

The sequence $(c_0,c_1,\dots,c_N)$ has no internal zeros. Indeed, each term $x^{\alpha_r}(1+x)^{\beta_r}$ has lowest degree term $x^{\alpha_r}$ which has degree at least $\alpha_m$, so $c_k=0$ for $k<\alpha_m$. On the other hand, the final summand $x^{\alpha_m}(1+x)^{\beta_m}$ has support on all degrees $\alpha_m,\alpha_m+1,\dots,\alpha_m+\beta_m$, and since $\alpha_m+\beta_m=n_m=N$, we obtain $c_k>0$ for all $\alpha_m\le k\le N$.

To prove unimodality, we show log-concavity:
\begin{equation}\label{eq:logconc}
c_k^2 \ge c_{k-1}c_{k+1}\qquad (0\le k\le N),
\end{equation}
interpreting $c_{-1}=c_{N+1}=0$. Log-concavity together with the absence of internal zeros implies unimodality (see \cite{Stanley1989}). We obtain \eqref{eq:logconc} by encoding the coefficients as path counts in a planar acyclic network and applying the Lindstr\"om--Gessel--Viennot lemma \cite{GesselViennot1985,Lindstrom1973} together with a planar noncrossing argument from total positivity (see \cite{Postnikov2006}).

Consider the directed grid graph on $\mathbb{Z}^2$ with edges $(i,j)\to(i+1,j)$ and $(i,j)\to(i,j+1)$ (steps east or north). For $0\le k\le N$, define sinks on the diagonal $x+y=N+1$ by
\[
B_k := (k,\,N-k+1).
\]
For $1\le r\le m$, define vertices
\[
U_r := (\alpha_r,\,m-r+1),
\qquad
V_r := (\alpha_r+1,\,m-r),
\]
and adjoin two additional sources $s$ and $t$, with directed edges $s\to U_r$ and $t\to V_r$ for each $r$. We embed this as a planar network in the evident way: the grid is planar, we place $s$ above $t$ to the left, and we draw the edges $s\to U_r$ and $t\to V_r$ as noncrossing arcs. Let $p(X,Y)$ denote the number of directed paths from $X$ to $Y$ in this network.

We compute the number of paths from $U_r$ and $V_r$ to $B_k$. A directed path from $U_r=(\alpha_r,m-r+1)$ to $B_k=(k,N-k+1)$ must take $E=k-\alpha_r$ east steps and
\[
N'=(N-k+1)-(m-r+1)=N-k-m+r
\]
north steps. Since $N=n_r+(m-r)$, we have $N'=n_r-k$. The total number of steps is then
\[
E+N'=(k-\alpha_r)+(n_r-k)=n_r-\alpha_r=\beta_r,
\]
and the number of such paths is $\binom{E+N'}{E}=\binom{\beta_r}{k-\alpha_r}$. Thus $p(U_r,B_k)=\binom{\beta_r}{k-\alpha_r}$. Similarly, a directed path from $V_r=(\alpha_r+1,m-r)$ to $B_k$ must take $E=k-\alpha_r-1$ east steps and
\[
N'=(N-k+1)-(m-r)=N-k-m+r+1=n_r-k+1
\]
north steps, so again the total number of steps is $\beta_r$ and $p(V_r,B_k)=\binom{\beta_r}{k-\alpha_r-1}$. Every path from $s$ to $B_k$ first chooses one edge $s\to U_r$ for some $r$ then follows a path from that $U_r$ to $B_k$, so
\[
p(s,B_k)=\sum_{r=1}^m p(U_r,B_k)=\sum_{r=1}^m \binom{\beta_r}{k-\alpha_r}=c_k
\]
by \eqref{eq:ck-sum}. Likewise,
\[
p(t,B_k)=\sum_{r=1}^m p(V_r,B_k)=\sum_{r=1}^m \binom{\beta_r}{k-\alpha_r-1}=c_{k-1},
\]
and similarly $p(s,B_{k+1})=c_{k+1}$ and $p(t,B_{k+1})=c_k$.

Now fix $k$ with $0\le k\le N-1$ and consider the $2\times 2$ determinant
\[
\Delta_k :=
\det\begin{pmatrix}
p(s,B_k) & p(s,B_{k+1})\\
p(t,B_k) & p(t,B_{k+1})
\end{pmatrix}
=
\det\begin{pmatrix}
c_k & c_{k+1}\\
c_{k-1} & c_k
\end{pmatrix}
= c_k^2 - c_{k-1}c_{k+1}.
\]
By the Lindstr\"om--Gessel--Viennot lemma for two sources and two sinks \cite{GesselViennot1985,Lindstrom1973}, $\Delta_k$ equals the number of vertex-disjoint path pairs $(s\to B_k,\ t\to B_{k+1})$ minus the number of vertex-disjoint path pairs $(s\to B_{k+1},\ t\to B_k)$. In our planar embedding, $s$ lies above $t$ on the left, while $B_k$ lies above $B_{k+1}$ on the right. A planar noncrossing argument shows that there are no vertex-disjoint path pairs connecting $s$ to the lower sink $B_{k+1}$ and $t$ to the upper sink $B_k$: such a pair would force a crossing in the plane, which is impossible for vertex-disjoint directed paths in this acyclic planar network. Consequently the second term vanishes and $\Delta_k\ge 0$. Therefore $c_k^2\ge c_{k-1}c_{k+1}$ for $0\le k\le N-1$, and with the boundary convention $c_{-1}=c_{N+1}=0$ the same holds for $k=0$ and $k=N$. Thus $(c_0,c_1,\ldots,c_N)$ is log-concave. The result follows since the sequence has no internal zeros.
\end{proof}

Lemma~\ref{lem:unimodal_staircase_binomial_sum} paves the way for proving that the domination polynomials of threshold graphs are unimodal. We first recall how this ubiquitous class of graphs is constructed. A graph $T$ on $n$ vertices is a \emph{threshold graph} if there is an ordering of the vertex set $v_1,\dots,v_n$ such that for each $j$ the vertex $v_j$ is either isolated or dominating in the induced subgraph on $\{v_1,\dots,v_{j-1},v_j\}$. We call such an ordering a \emph{threshold ordering}. In a threshold ordering, the vertices added as dominating themselves form a clique and the vertices added as isolated form an independent set. We refer to these vertices throughout as vertices on the dominating side and vertices on the independent side, respectively. The structure of neighborhoods is closely connected to these sets. For instance, neighborhoods from the independent side into the clique side are nested: if $u$ and $w$ are vertices on the independent side with $u$ earlier than $w$ in the ordering, and $C$ is the set of vertices on the dominating side, then $N(u)\cap C\supseteq N(w)\cap C$. See \cite{MahadevPeled1995} for these standard properties.

Fix a threshold ordering $v_1,\dots,v_n$. List the vertices on the dominating side, in the order they appear, as $c_1,\dots,c_m$. Write $t_r$ for the index with $c_r=v_{t_r}$. For each $r$ write $P_r=\{v_1,\dots,v_{t_r-1}\}$, the vertices in $T$ appearing before $c_r$ does. Let $I$ be the set of vertices that were added on the isolated side. For each $r$, let $I_r\subseteq I$ be the set of vertices on the isolated side that appear after $c_r$ in the threshold ordering, so $I_r=\{v_j\in I:j>t_r\}$. The next lemma shows that dominating sets can be partitioned by their last vertex on the dominating side in the threshold ordering.

\begin{lemma}\label{lem:threshold-decomp}
Let $T$ be a threshold graph with a fixed threshold ordering as above. Every dominating set $S$ that contains at least one vertex from the dominating side of the threshold ordering of $T$ has a unique index $r$ such that $c_r$ is the last vertex on the dominating side of the threshold ordering that lies in $S$, and for this $r$ one has
\[
S=U \cup \{c_r\}\cup I_r
\]
for an arbitrary subset $U\subseteq P_r$. A dominating set of $T$ that contains no vertex from the dominating side of the threshold ordering exists if and only if there is at least one vertex on the isolated side of $T$ before $c_1$, and in that case the only such dominating set is the set of all vertices on the isolated side.
\end{lemma}


To illustrate Lemma~\ref{lem:threshold-decomp}, consider the threshold graph $T$ with vertices $v_1,v_2,\ldots,v_{14}$ listed in threshold order, whose dominating side consists of the vertices $\{v_3,v_5,v_9,v_{11}\}$, so $c_1=v_3, \ c_2=v_5, \ c_3=v_9,$ and $v_{11}=c_4$. The graph $T$ is illustrated in Figure~\ref{fig:four-prefix-joins}.

\begin{figure}[h]
\centering
\begin{tikzpicture}[
  x=0.85cm, y=1cm,
  vertex/.style={circle, fill=black, inner sep=0pt, minimum size=3.5pt},
  edge/.style={line width=0.45pt}
]

  \foreach \i in {1,...,14}{
    \node[vertex, label=below:{\scriptsize $v_{\i}$}] (v\i) at (\i,0) {};
  }

  \foreach \i in {3,5,9,11}{
    \pgfmathtruncatemacro{\imax}{\i-1}
    \foreach \j in {1,...,\imax}{
      \pgfmathsetmacro{\loose}{0.55 + 0.12*(\i-\j)}
      \draw[edge] (v\j.north) to[out=60,in=120,looseness=\loose] (v\i.north);
    }
  }

\end{tikzpicture}
\caption{A threshold graph with threshold order $v_1,v_2,\ldots,v_{14}$, whose dominating side is $v_3,v_5,v_9,v_{11}$.}
\label{fig:four-prefix-joins}
\end{figure}
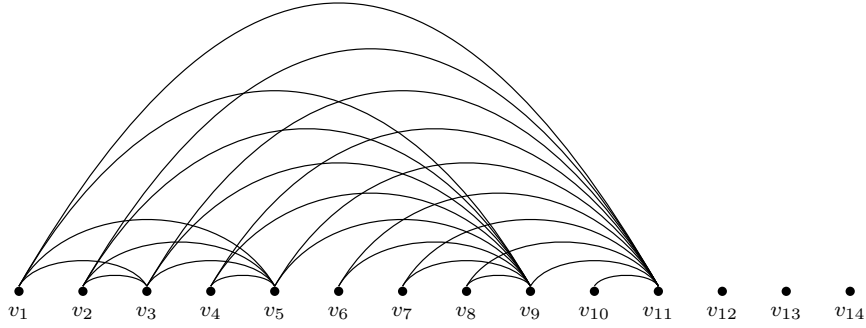

An example of a dominating set is $S=\{v_3,v_4,v_5,v_8,v_9,v_{10},v_{12},v_{13},v_{14}\}$. Notice that $c_3=v_9$ is the last vertex in $S$ on the dominating side (in threshold order). So we if decompose $S$ as $\{v_3,v_4,v_5,v_8\} \cup \{v_9\} \cup \{v_{10},v_{12},v_{13},v_{14}\}$ we have $S$ is of the form $U \cup \{c_3\} \cup I_3$ because $I_3=\{v_{10},v_{12},v_{13},v_{14}\}$. The vertices in $I_3$ have to be in $S$ because $c_3=v_9$ is the last vertex on the dominating side that is in $S$, so all vertices after it in threshold order that are on the independent side must be in it. Moreover $U \subseteq \{v_1,v_2,\ldots,v_8\} = P_3$. Replacing $U$ with any subset of $P_3$ will keep $S$ dominating because $c_3=v_9$ is adjacent to all vertices in $P_3$.

\begin{proof}[Proof of Lemma~\ref{lem:threshold-decomp}]
Let $S$ be a dominating set containing at least one vertex on the dominating side and let $c_r$ be the last vertex on the dominating side in the threshold ordering that is also in $S$. Consider a vertex $w$ on the isolated side that appears after $c_r$ in threshold order. At the moment $w$ is added in the threshold construction it is isolated from the previously built graph, so $w$ has no neighbors among $\{v_1,\dots,v_{t_r}\}$. Since all other vertices on the dominating side appear after $c_r=v_{t_r}$, and $c_r$ is the last vertex on the dominating side in $S$, the set $S$ contains no vertex on the dominating side that dominates $w$. Thus $w$ must belong to $S$. This shows $I_r\subseteq S$ and also $c_r\in S$.

Every vertex in the prefix $P_r$ is adjacent to $c_r$ because $c_r$ was added as a vertex on the dominating side, so $c_r$ dominates all vertices of $P_r$. Every vertex $c_s$ on the dominating side with $s>r$ is adjacent to $c_r$ because each vertex on the dominating side is adjacent to all earlier vertices. Thus every vertex outside $I_r$ is dominated by $c_r$. It follows that any set of the form $\{c_r\}\cup I_r\cup U$ with $U\subseteq P_r$ is dominating. Conversely, we have already shown that a dominating set $S$ with last vertex $c_r$ on the dominating side must contain $\{c_r\}\cup I_r$, and any additional vertices of $S$ must come from $P_r$. Therefore $S=\{c_r\}\cup I_r\cup U$ for some $U\subseteq P_r$. The index $r$ is unique because the last vertex on the dominating side in a given set is unique.

If $S$ contains no vertex on the dominating side, then $S\subseteq I$. Any vertex in $I\setminus S$ has no neighbor in $I$, so this forces $S$ to be $I$. The set $I$ dominates the first vertex $c_1$ on the dominating side if and only if $c_1$ has at least one neighbor in $I$, and in a threshold ordering $c_1$ is adjacent precisely to the vertices that appear before it. Hence $I$ is dominating if and only if there exists a vertex on the isolated side before $c_1$.
\end{proof}

Lemma \ref{lem:threshold-decomp} yields a closed form for the domination polynomial of a threshold graph.

\begin{lemma}\label{lem:closedform}
Let $T$ be a threshold graph. Let $c_1,\ldots,c_m$, $I$, $I_1,\ldots,I_m$ and $P_1,\ldots,P_m$ be as in Lemma~\ref{lem:threshold-decomp}. Then
\begin{equation}\label{eq:threshold-formula}
D(T,x)=\sum_{r=1}^m x^{\alpha_r}(1+x)^{\beta_r}+\varepsilon x^{|I|}.
\end{equation}
where for each $1 \leq r \leq m$, $\beta_r=|P_r|$, $\alpha_r=|I_r|+1$, and $\varepsilon \in \{0,1\}$ is $1$ if $I$ is a dominating set and $0$ otherwise.
\end{lemma}
\begin{proof}
By Lemma~\ref{lem:threshold-decomp}, 
\begin{align*}
D(T,x) &=\sum_{S \text{ a dominating set }} x^{|S|} = \sum_{r=1}^m \sum_{U \subseteq P_r} x^{|\{c_r\} \cup I_r \cup U|} + \varepsilon x^{|I|} \\
&=\sum_{r=1}^m x^{|\{c_r \cup I_r\}|} \left(\sum_{U \subseteq P_r} x^{|U|}\right) + \varepsilon x^{|I|} =\sum_{r=1}^m x^{\alpha_r}(1+x)^{\beta_r}+\varepsilon x^{|I|}.
\end{align*}
\end{proof}

The parameters in \eqref{eq:threshold-formula} satisfy a rigid staircase condition. Write $n_r=\alpha_r+\beta_r$ for the degree of the $r$th binomial summand. Moving from $c_r$ to $c_{r+1}$ shifts some vertices on the isolated side from the forced set $I_r$ into the optional prefix $P_{r+1}$ and adds one additional optional vertex corresponding to $c_r$ itself entering the prefix. Concretely, if there are $s\ge 0$ on the isolated side between $c_r$ and $c_{r+1}$ in the threshold ordering, then $\beta_{r+1}=\beta_r+s+1$ and $\alpha_{r+1}=\alpha_r-s$. Adding yields
\begin{equation}\label{eq:degree-step}
n_{r+1}=n_r+1.
\end{equation}


We can now conclude unimodality for threshold graphs.

\begin{theorem}\label{thm:thresholdUnimodal}
If $T$ is a threshold graph, then its domination polynomial $D(T,x)$ is unimodal.
\end{theorem}

\begin{proof}
Fix a threshold ordering and write $D(T,x)$ in the form \eqref{eq:threshold-formula} from Lemma~\ref{lem:closedform}, using the same notation therein for the involved sets. If there are no vertices on the dominating side then $D(T,x)=\varepsilon x^{|I|}$, which is a single monomial and hence unimodal, so assume $m\ge 1$. By construction one has $\alpha_1\ge\alpha_2\ge\cdots\ge\alpha_m\ge 1$ and $0\le \beta_1<\beta_2<\cdots<\beta_m$, and the degrees $n_r:=\alpha_r+\beta_r$ satisfy the staircase condition \eqref{eq:degree-step}. If $\varepsilon=0$, then $D(T,x)=\sum_{r=1}^{m} x^{\alpha_r}(1+x)^{\beta_r}$ is exactly of the form required by Lemma~\ref{lem:unimodal_staircase_binomial_sum}, and therefore has a unimodal coefficient sequence. If $\varepsilon=1$, define an additional pair $(\alpha_0,\beta_0):=(|I|,0)$ so that $\varepsilon x^{|I|}=x^{\alpha_0}(1+x)^{\beta_0}$. In this case Lemma~17 implies there is at least one vertex on the isolated side before $c_1$, hence $P_1\subseteq I$ is nonempty and $\beta_1=|P_1|\ge 1$, so $0=\beta_0<\beta_1$. Moreover $I$ is the disjoint union of $P_1$ and $I_1$, so $n_1=\alpha_1+\beta_1=(|I_1|+1)+|P_1|=|I|+1=n_0+1$. Together with \eqref{eq:degree-step} this gives $n_{r+1}=n_r+1$ for $r=0,\dots,m-1$, and we also have $\alpha_0=|I|\ge |I_1|+1=\alpha_1$ (hence $\alpha_0\ge\alpha_1\ge\cdots\ge\alpha_m$). Consequently, $D(T,x)=\sum_{r=0}^{m} x^{\alpha_r}(1+x)^{\beta_r}$ satisfies the hypotheses of Lemma~\ref{lem:unimodal_staircase_binomial_sum}, and therefore the coefficient sequence of $D(T,x)$ is unimodal.
\end{proof}

\section{Conclusion and further directions}

The overaching theme in this article is that tools from outside of domination polynomial literature can pave the way for new and rich perspectives that expand our understanding of these polynomials, particularly in the lens of resolving their unimodality. The three perspectives developed in this paper point toward several avenues for extending these results while keeping in line with this overarching theme.

\medskip\noindent
\textbf{Perspective 1: Beyond closed neighborhoods.}
The identity \eqref{eq:reciprocity} is a manifestation of a general principle: domination is a transversal condition in an associated hypergraph. Many variations in domination theory admit analogous hypergraph models, for example total domination (hitting open neighborhoods), distance-$r$ domination (hitting $r$-neighborhoods), and domination in directed graphs. Each gives rise to a natural polynomial that enumerates dominating objects. It would be interesting to systematically develop reciprocity identities and to identify which zero-free regions for hypergraph independence polynomials can be transferred to these variants. In particular, one can ask for analogues of Theorems~\ref{thm:linearRootBound} and \ref{thm:rayZeroFree} for total domination polynomials and for distance domination polynomials in classes of bounded degree.

\medskip\noindent
\textbf{Perspective 2: Higher order overlap corrections.}
The spanning tree correction in Proposition~\ref{prop:overlapBound} is a controlled inclusion--exclusion refinement that subtracts a collection of pairwise overlaps. In graphs with significant neighborhood repetition, such as graphs with large twin classes or nested neighborhoods, one expects higher-order overlaps to also be large. This suggests developing corrections based on larger structures (for example, bounded depth trees, hypertrees, or sparse complexes) that capture triple and higher intersections among the events $A_v$. Understanding the best possible overlap corrections for a given graph family, and how they interact with the Beaton-Brown criterion, is a promising direction for producing unimodality results in classes where the minimum degree is small but the neighborhood structure is highly redundant.

\medskip\noindent
\textbf{Perspective 3: Total positivity beyond threshold graphs.}
The threshold graph proof proceeds by decomposing dominating sets into disjoint families whose generating functions are shifted binomial rows forming a staircase. This invites the search for other graph classes admitting similar decompositions. A natural candidate for instance is split graphs, which are generalizations of threshold graphs. In such graphs it is plausible that dominating sets can be partitioned by the location of a ``last'' vertex in an appropriate ordering, producing sums of binomial pieces with controlled overlap. The algebraic Lemma~\ref{lem:unimodal_staircase_binomial_sum} is robust and should apply whenever such a staircase structure can be established. Furthermore, it is possible that generalizations of Lindstr\"{o}m-Gessel-Viennot such as that in the work of Hopkins and Zaimi (see \cite{HopkinsZaimi2023}) can resolve unimodality for more general classes beyond threshold graphs.

\medskip\noindent
\textbf{Toward hereditary unimodality results.}
Conjecture~\ref{conj:unimodal} is open even for trees. The three perspectives suggest three different approaches to hereditary classes: through the analytic properties of the closed neighborhood hypergraph, through probabilistic estimates on neighborhood-avoidance events at $k\approx n/2$, and through explicit decompositions into unimodal polynomial pieces. It would be interesting to understand whether these approaches can be combined. For example, chordal graphs and interval graphs often admit elimination orderings that enforce nested neighborhood structure, and they also have closed neighborhood hypergraphs with constrained intersection patterns. Whether these features can be leveraged to build staircase decompositions or effective overlap corrections, and thereby prove unimodality for new hereditary classes, remains an appealing open problem.

\medskip\noindent
\textbf{Concrete questions.}
We close with a few concrete questions motivated by the methods above.

\smallskip\noindent
(1) For a fixed $\Delta$, what is the best possible region in the complex plane that is guaranteed to be zero-free for all domination polynomials of graphs with maximum degree at most $\Delta$? The transfer from \cite{BencsBuys2025} is optimal for general bounded degree hypergraphs, but the closed neighborhood hypergraphs $\HH_G$ satisfy strong additional constraints. Can those constraints be exploited to enlarge the zero-free region beyond what is possible for arbitrary hypergraphs of the same maximum degree?

\smallskip\noindent
(2) The quantity $\tau_k(G)$ is defined through a maximum spanning tree problem and is therefore efficiently computable. Can one identify additional, efficiently computable corrections that incorporate higher-order overlaps
and yield substantially stronger unimodality criteria on sparse graph classes such as trees?

\smallskip\noindent
(3) How much can total positivity afford us? For instance, is there a way to prove split graphs have unimodal domination polynomials using total positivity? Split graphs generalize threshold graphs by relaxing the nested neighborhood condition, so an affirmative answer would extend Theorem~\ref{thm:thresholdUnimodal} to a significantly larger hereditary class. If not, can one use generalizations of Lindstr\"{o}m-Gessel-Viennot such as in \cite{HopkinsZaimi2023} to apply total positivity beyond threshold graphs?

\section{Acknowledgments}
The author is partially funded by research funds from York University, and NSERC Discovery Grant \#RGPIN-2025-06304.

\end{document}